\documentclass[pdf,preprint,12pt]{imsart}
\usepackage[margin=1in]{geometry}
\RequirePackage{natbib}
\RequirePackage[OT1]{fontenc}
\RequirePackage{amsthm,amsmath}
\RequirePackage{natbib}
\RequirePackage{hyperref}
\usepackage{bbm}
\usepackage{amsmath,amssymb, amsthm}
\usepackage{amsfonts}
\numberwithin{equation}{section}
\theoremstyle{plain}
\newtheorem{thm}{Theorem}[section]
\newtheorem{corollary}{Corollary}[section]

\newtheorem{lemma}{Lemma}[section]
\newtheorem{proposition}{Proposition}[section]
\newtheorem{definition}{Definition}
\newtheorem*{definition*}{Definition}
\newtheorem{remark}{Remark}[section]
\newcommand{\m}[1]{\boldsymbol{#1}}

\newcommand{\Z}{\mathcal{Z}}

\newcommand{\E}{\mathbb{E}}
\newcommand{\PP}{\mathbb{P}}
\newcommand{\II}{\mathbbm{1}}
\newcommand{\ol}{\overline}

\newcommand{\field}[1]{\mathbb{#1}}
\newcommand{\R}{\field{R}}

\newcommand{\HH}{\field{H}}
\newcommand{\MM}{\field{M}}

\newcommand{\Propp}{\textbf{P}}

\DeclareMathOperator{\Tr}{tr}
\DeclareMathOperator{\Tra}{\ol{tr}}

\begin{document}

\begin{frontmatter}
\title{A Note on Matrix Concentration Inequalities via the Method of Exchangeable Pairs}
\runtitle{Matrix concentration inequalities}

\begin{aug}
\author{\fnms{Daniel} \snm{Paulin} \ead[label=e1]{paulindani@gmail.com}} 
\runauthor{D. Paulin}

\affiliation{National University of Singapore}

\address{
Department of Mathematics, National University of Singapore\\
10 Lower Kent Ridge Road, Singapore 119076, Republic of Singapore.\\
\printead{e1}}

\end{aug}



\begin{abstract}
The aim of this paper is to prove an improved version of the bounded differences inequality for matrix valued functions (see \cite{Troppuf}, Corollary 7.5), by developing the methods of \cite{Mackey}. Along the way, we prove new trace inequalities for the matrix exponential.
\end{abstract}

\end{frontmatter}

\section{Introduction}
Let $Z_1,\ldots,Z_n$ be independent (or dependent) random variables, and $\m{X}=f(Z_1,\ldots,Z_n)$ be a random Hermitian matrix. One specific example is when $\m{X}=\sum_{k}\m{X}_k$ is a sum of random matrices.  In many situations, we are interested in bounding the quantity $\PP(\lambda_{max}(\m{X})\ge t)$.

\cite{ahlswede2002strong} was the first to use Laplace transform method in this setting, they show that for any random Hermitian matrix $\m{X}$,
\begin{equation}
\PP(\lambda_{max}(X)\ge t)\le \inf_{\theta>0}\left\{e^{-\theta t}\E \Tr \exp(\theta \m{X})\right\},
\end{equation}
thus for $\m{X}=\sum_{k} \m{X}_k$,
\begin{equation}
\PP(\lambda_{max}(X)\ge t)\le \inf_{\theta>0}\left\{e^{-\theta t}\E \Tr \exp\left(\theta \sum_{k} \m{X}_k\right)\right\}.
\end{equation}
Estimating the right hand side now poses a difficulty, because in general, $e^{A+B}\neq e^{A}\cdot e^{B}$ for the matrix exponential.

\cite{Troppuf} proves the following lemma to estimate the right hand side:
\begin{lemma}[Lemma 3.4 of \cite{Troppuf}]
Consider a finite sequence $\{\m{X}_k\}$ of independent, random, self-adjoint matrices. Then
\[\E \Tr \exp\left(\sum_k \theta \m{X}_k \right)\le \Tr \exp\left(\sum_k \log \E e^{\theta \m{X}_k} \right) \text{ for } \theta\in \R.\]
\end{lemma}

This Lemma is based on a corollary of \cite{Troppuf} (which is derived from a Theorem of Lieb):
\begin{corollary}[Corollary 3.3 of \cite{Troppuf}] Let $\m{H}$ be a fixed self-adjoint matrix, and let $\m{X}$ be a random self-adjoint matrix. Then
\[\E \Tr \exp (\m{H}+\m{X})\le \Tr \exp\left(\m{H}+\log\left(\E e^{\m{X}}\right)\right).\]
\end{corollary}

These inequalities are used in \cite{Troppuf} to prove matrix versions of various concentration inequalities for sums of random matrices (Chernoff, Bernstein), and inequalities for matrix martingales (Azuma-Hoeffding, and matrix bounded differences).

\cite{Mackey} takes a different approach. They make the following basic definition(\cite{Mackey}, Definition 2.2):
\begin{definition}[Matrix Stein Pair]
Let $(Z,Z')$ be an exchangeable pair of random variables taking values in a polish space $\Z$, and let $\m{\Psi}:\Z \to \HH^d$ be a measurable function. Define the random Hermitian matrices 
\[\m{X}:=\Psi(Z) \hspace{5mm} \text{and} \hspace{5mm} \m{X}':=\Psi(Z').\]
We say that $(\m{X},\m{X}')$ is a \emph{matrix Stein pair} if there is a constant $\alpha\in (0,1]$ for which
\begin{equation}
\E\left(\m{X}-\m{X}'|Z\right)=\alpha \m{X} \text{ almost surely.}
\end{equation}
The constant $\alpha$ is called the \emph{scale factor} of the pair. When discussing a matrix Stein pair $(\m{X},\m{X}')$, we always assume that $\E ||\m{X}||^2<\infty$.
\end{definition}

Suppose that $(\m{X},\m{X}')$ is a matrix Stein pair, then they write the derivate moment generating function of $m(X)$ as
\begin{eqnarray*}
m(\theta)'&=&\E\Tr\left(\m{X} e^{\theta \m{X}}\right)=\E \Tr \left(\frac{1}{\alpha}(\m{X}-\m{X}') e^{\theta \m{X}}\right)\\&=&\E\Tr \left(\frac{1}{2\alpha}(\m{X}-\m{X}') \left(e^{\theta \m{X}}-e^{\theta \m{X}'}\right)\right),
\end{eqnarray*}
using exchangeability in the last step.

To further bound this quantity, they prove the following trace inequality (Lemma 3.4 of \cite{Mackey}):
\begin{lemma}
Let $I$ be an interval of the real line. Suppose that $g: I\to \R$ is a weakly increasing function and that $h: I \to \R$ is a function whose derivative $h'$ is convex. For all matrices 
$\m{A},\m{B}\in \HH^{d}(I)$,
\begin{eqnarray*}
&&\Tr[(g(\m{A})-g(\m{B}))\cdot (h(\m{A})-h(\m{B})]\\
&&\le\frac{1}{2}\Tr\left[(g(\m{A})-g(\m{B}))(\m{A}-\m{B})\cdot (h'(\m{A})+h'(\m{B})) \right].
\end{eqnarray*}
When $h'$ is concave, the inequality is reversed. The same results hold for the standard trace.
\end{lemma}
This lemma is based on a standard trace inequality (\cite{Petzsurvey}, Proposition 3).
\begin{corollary}\label{cortrexp}
For $\theta>0$,
\[\Tr \left((\m{X}-\m{X}') \left(e^{\theta \m{X}}-e^{\theta \m{X}'}\right)\right)\le \frac{\theta}{2}
\Tr \left((\m{X}-\m{X}')^2 \left(e^{\theta \m{X}}+e^{\theta \m{X}'}\right)\right)
\]
\end{corollary}

Using this corollary, we can bound the derivate of the trace mgf:
\begin{eqnarray*}
m(\theta)'&\le&  \E\Tr \left(\frac{\theta}{2\alpha}(\m{X}-\m{X}')^2 \left(e^{\theta \m{X}}+e^{\theta \m{X}'}\right)\right)\\
&=& \E\Tr \left(\frac{\theta}{\alpha}(\m{X}-\m{X}')^2 e^{\theta \m{X}}\right),
\end{eqnarray*}
and this quantity can be bounded in many situations.

The advantage of this approach compared to \cite{Troppuf} is that the constants are often better, and some dependent cases can be also treated. The disadvantage is that other than sums of random matrices, few other cases can be written as Stein pairs. This means that matrix martingales, and the method of bounded differences, are not possible to recover.

The purpose of this paper is to show that \cite{Mackey} can be improved to show the method of bounded differences for matrix valued functions. We are going to prove new trace inequalities, which generalize Corollary \ref{cortrexp}, and allow us to go beyond Stein pairs.

Our inequality also works for weakly dependent random variables. We quantify the dependence by a matrix:
\begin{definition*}[Dobrushin's interdependence matrix]
Let $X:=(X_1,\ldots,X_n)$ be a random vector taking values in $\Lambda:=(\Lambda_1,\ldots,\Lambda_n)$, with law $\mu$. Suppose that $D:=(d_{ij})_{1\le i,j\le n}$ is an $n\times n$ matrix with nonnegative entries and zeroes on the diagonal such that for any $i$, and any $x,y\in \Lambda$, 
\[d_{TV}(\mu_i(\cdot | x_{-i}),\mu_i(\cdot | y_{-i}))\le \sum_{j=1}^n d_{ij}\II[x_j\ne y_j]\]
where $d_{TV}$ is the total variational distance. 
Then we say that $D$ is a Dobrushin interdependence matrix for the random vector $X$ (or equivalently random measure $\mu$). Here $x_{-i}:=(x_1,\ldots,x_{i-1},x_{i+1},\ldots,x_n)$ and $\mu_i(\cdot | x_{-i})$ is the conditional distribution of $X_i$ given $X_{-i}=x_{-i}$.
\end{definition*}

Concentration inequalities for real valued functions Hamming Lipschitz functions under the condition $||D||_2<1$ have been proven in \cite{Cth}, Chapter 4.

\section{Results}

The following result is a strengthening of Corollary 7.5 of \cite{Troppuf}. We have exponent $-t^2/\sigma^2$ instead of $-t^2/8\sigma^2$ in the independent case, and our result also works under Dobrushin-type weak dependence.

\begin{thm}\label{thmboundediff}
Let $\{Z_k:k=1,\ldots,n\}$ be an independent family of random variables, and let $\m{H}$ be a function that maps $n$ variables to a self adjoint matrix of dimension $d$.
Consider a sequence $\{\m{A}_k\}$ of fixed self-adjoint matrices that satisfy
\begin{equation}\label{Hammingcond}
\left(\m{H}(z_1,\ldots,z_k,\ldots,z_n)-\m{H}(z_1,\ldots,z_k',\ldots,z_n)\right)^2\le \m{A}_k^2,
\end{equation}
where $z_i$ and $z_i'$ range over all possible values of $Z_i$ for each index $i$. Compute the variance parameter
\begin{equation}
\sigma^2:=\left|\left |\sum_k \m{A}_k^2\right|\right|.
\end{equation}
Then for all $t\ge 0$,
\begin{equation}
\PP\left\{\lambda_{max}\left(\m{H}(Z)-\E \m{H}(Z)\right)\ge t\right\}\le d\cdot e^{-t^2/\sigma^2}
\end{equation}
where $Z=(Z_1,\ldots,Z_n)$.

Alternatively, suppose that $\{Z_k:k=1,\ldots,n\}$ is a family of dependent random variables with Dobrushin interdependence matrix $D$. If $D$ satisfies $\max(||D||_1,||D||_{\infty})<1$, then for every $t\ge 0$,
\begin{equation}
\PP\left\{\lambda_{max}\left(\m{H}(Z)-\E \m{H}(Z)\right)\ge t\right\}\le d\cdot e^{-t^2/(c\sigma^2)},
\end{equation}
with 
\begin{equation}\label{Ceq}
c:=\frac{1/(1-||D||_1)+1/(1-||D||_{\infty})}{2}.
\end{equation}
\end{thm}

A simple corollary of this the following matrix Hoeffding bound (in the independent case, similar to Corollary 4.2 of \cite{Mackey}):

\begin{corollary}\label{corboundediff}
Let $\{\m{Y}_k:k=1,\ldots,n\}$ be an independent family of $\HH^d$ matrices, and let $\m{H}$ be a function that maps $n$ variables to a self adjoint matrix of dimension $d$.
Consider a sequence $\{\m{A}_k\}$ of fixed self-adjoint matrices, 
\begin{equation}\label{Hammingcond2}
\E\m{Y}_k=0, \hspace{5mm} \m{Y}_k^2\preceq \m{A}_k^2,
\end{equation}
Define the variance parameter
\[\sigma^2:=\left|\left |\sum_k \m{A}_k^2\right|\right|.\]
Then for all $t\ge 0$,
\[\PP\left\{\lambda_{max}\left(\sum_k \m{Y}_k\right)\ge t\right\}\le d\cdot e^{-t^2/(4\sigma^2)}.\]
Alternatively, for $\{\m{Y}_k:k=1,\ldots,n\}$ weakly dependent with Dobrushin matrix $D$ satisfying $\max(||D||_1,||D||_{\infty})<1$, we have
\[\PP\left\{\lambda_{max}\left(\sum_k \m{Y}_k\right)\ge t\right\}\le d\cdot e^{-t^2/(4c \sigma^2)},\]
with $c$ defined as in \eqref{Ceq}.
\end{corollary}
\begin{remark}
The 4 in the exponent  comes from the fact that \eqref{Hammingcond} is satisfied for $2\m{A}_k$.
\end{remark}

An important tool in the proof is the following trace inequality:

\begin{thm}\label{thmABC}
Let $A,B,C$ be Hermitian matrices of equal size, then
\begin{eqnarray*}
\Tr\left(C(e^A-e^B)\right)\le \Tr\left(\frac{C^2+(A-B)^2}{2}\left(\frac{e^{A}+e^{B}}{2}\right)\right).
\end{eqnarray*} 
\end{thm}
\begin{corollary}\label{corABC}
Under the same conditions, for $\theta>0$,
\begin{eqnarray*}
\Tr\left(C(e^{\theta A}-e^{\theta B})\right)\le \theta \Tr\left(\frac{C^2+(A-B)^2}{2}\left(\frac{e^{\theta A}+e^{\theta B}}{2}\right)\right),
\end{eqnarray*} 
and for $\theta<0$,
\begin{eqnarray*}
\Tr\left(C(e^{\theta A}-e^{\theta B})\right)\ge \theta \Tr\left(\frac{C^2+(A-B)^2}{2}\left(\frac{e^{\theta A}+e^{\theta B}}{2}\right)\right).
\end{eqnarray*} 
\end{corollary}
\begin{proof}
Apply Theorem \ref{thmABC} to $\theta A,\theta B, \theta C$.
\end{proof}
We also prove this result:

\begin{thm}\label{mxholder}(Matrix H\"older inequality)
Let $A,B,C,D$ be Hermitian matrices with $A$ and $B$ positive semidefinite, and $0\le p\le 1$, then we have
\begin{equation}\label{holdereq2}
Re\left(\Tr\left(CA^{p}D B^{1-p}+CA^{1-p}D B^p\right)\right)\le  \Tr 
\left(\frac{C^2+D^2}{2}\left(A+B\right)\right).
\end{equation}
\end{thm}

\section{Proof of the bounded differences inequality}
For a random matrix $\m{X}$, the normalized trace mgf is defined, similarly to Definition 3.2 of \cite{Mackey},as
\[m(\theta):=m_{\m{X}}(\theta)=\E \Tra e^{\theta \m{X}}=\frac{1}{d}\E \Tr e^{\theta \m{X}},\]
which may not exists for all values of $\theta$.

We are going to use Proposition 3.3 of \cite{Mackey}
\begin{proposition}\label{mxlapltf}(Matrix Laplace Transform Method)
Let $\m{X}$ be a random matrix with normalized trace mgf $m(\theta):=\E \Tra e^{\theta \m{X}}$. For each $t\in \R$,
\begin{eqnarray}
\label{eqlapl1}&&\PP\{\lambda_{max}(\m{X})\ge t\}\le d\cdot \inf_{\theta>0}\exp\{-\theta t + \log m(\theta)\}, \text{ and }\\
\label{eqlapl2}&&\PP\{\lambda_{min}(\m{X})\le t\}\le d\cdot \inf_{\theta<0}\exp\{-\theta t + \log m(\theta)\}.
\end{eqnarray}
\end{proposition}

\begin{proof}[Proof of Theorem \ref{thmboundediff}]
We follow the Markov chain approach \cite{Cth}.

As shown in Chapter 4, an exchangeable pair $(X,X')$ automatically defines a reversible Markov kernel
$P$ as
\[P f(X):= \E (f(X') |X=x),\]
where $f$ is any function with $\E |f(X)|<\infty$. Suppose that $X$ takes values in a Polish space $\Omega$, then 
\begin{lemma}[Lemma 4.1 of  \cite{Cth}] Suppose that $f:\Omega\to \R$ is a measurable function with $\E f(X)=0$, and there is a finite constant $L$ such that
\begin{equation}\label{eqcondpk}
\sum_{k=0}^{\infty} \left|P^k f(x)-P^k f(y)\right|\le L \text{ for every }x\text{ and }y,
\end{equation}
then 
\begin{equation}
F(x,y):=\sum_{k=0}^{\infty}\left(P^k f(x)-P^k f(y)\right)
\end{equation}
satisfies $F(X,X')=-F(X',X)$ and $\E(F(X,X')|X)=f(X)$.
\end{lemma}
With a simple adaptation of the proof, the reader can verify that this Lemma also holds for matrix valued functions $f:\Omega\to \HH^d$, with \eqref{eqcondpk} replaced by 
\begin{equation}\label{eqcondpknorm}
\sum_{k=0}^{\infty} \left|\left|P^k f(x)-P^k f(y)\right|\right|\le L \text{ for every }x\text{ and }y.
\end{equation}

We need to define property \/\Propp\/ as in \cite{Cth}:
\begin{definition*}
Let $\{X(k)\}_{k\ge 0}$ and $\{X'(k)\}_{k\ge 0}$ be two chains from the kernel defined by $(X,X')$, for arbitrary initial values $x,y\in \Omega$. We say that a coupling of these two chains satisfies property \/\Propp\/ if for every $x,y\in \Omega$, and every $k$, the marginal distribution of $X(k)$ only depends on $x$, and the marginal distribution of $X'(k)$ only depends on $y$.
\end{definition*}
We propose the following matrix version Lemma 4.2 of \cite{Cth} (the proof can be easily adapted):
\begin{lemma}
Suppose that a coupling of $\{X(k)\}_{k\ge 0}$ and $\{X'(k)\}_{k\ge 0}$ satisfies property \/\Propp\/. Let
$f:\Omega\to \HH^d$ be a function such that $\E f(X)=0$. Suppose that there exists $L<\infty$ such that for every $x,y\in \Omega$, 
\begin{equation}\label{sumcond}
\sum_{k=0}^{\infty} \left|\left|\E\left(f(X(k))-f(X'(k)) |X(0)=x,X'(0)=y\right)\right|\right|\le L.
\end{equation}
Then, the function $F$ defined as
\begin{equation}
F(x,y):=\sum_{k=0}^{\infty}\E\left(f(X(k))-f(X'(k))|X(0)=x,X'(0)=y\right)
\end{equation}
satisfies $F(X,X')=-F(X',X)$ and $\E(F(X,X')|X)=f(X)$.
\end{lemma}

First, we will prove the independent case:
\begin{proof}[Proof of independent case]
Let $X:=(X_{1},\ldots,X_{n})$ be a vector with independent components
($X_i$: component $i$,  $\{X(k)\}_{k\ge 0}$: Markov chain).

Let
\begin{equation}\label{X1nr}
X^{(r)}:=(X_{1}^{(r)},\ldots,X_{n}^{(r)})
\end{equation}
be independent copies of $X$, for $r\ge 0$. Let $I, I_1, \ldots, I_k\ldots$ be uniformly distributed indexes in $[n]$, independent of each other and of $X$ and $X^{(r)}$. Define $X'$ as
\[X'_{i}=X_{i}\text{ for }  i\ne I \text{ and } X'_{I}=X_{I}^{(0)}.\]
Now we are ready to construct $X(k)$ and $X'(k)$:\\
Suppose that $X(0)=x$ and $X'(0)=y$, for $x,y\in \Omega$. For $k\ge 1$, define $X(k)$ as
\[X_i(k):=X_i(k-1)\text{ for }  i\ne I_k \text{ and } X_{I_k}(k):=X_{I_k}^{(k)}.\]
Similarly, for $k\ge 1$, define $X'(k)$ as
\[X'_i(k):=X'_i(k-1)\text{ for }  i\ne I_k \text{ and } X'_{I_k}(k):=X_{I_k}(k).\]
With this definition, $\{X(k)\}_{k\ge 0}$ and $\{X'(k)\}_{k\ge 0}$ are having the same distribution as the Markov chain defined by the kernel
\/$P f(X)$, moreover $(X(k), X'(k))$ satisfy property \/\Propp. In practice, we will start with $X_0=X$ and $X_0'=X'$.

We can prove condition \eqref{sumcond} by the coupon collector's problem.

For this chain, we can write 
\begin{eqnarray*}
&&m(\theta)'=\E\Tr\left(f(X) e^{\theta f(X)}\right)=\E \Tr\left(F(X,X') \cdot e^{\theta f(X)} \right)\\
&&= \frac{1}{2}\E \Tr\left(\sum_{k=0}^{\infty} \left(f(X(k))-f(X'(k))\right) \cdot \left(e^{\theta f(X)}-e^{\theta f(X')}\right) \right)
\\
&&= \frac{1}{2}\sum_{k=0}^{\infty} \E \Tr\left( \II[I\notin I_1,\ldots,I_k]\left(f(X(k))-f(X'(k))\right) \cdot \left(e^{\theta f(X)}-e^{\theta f(X')}\right) \right)
\end{eqnarray*}
Now, using Theorem \ref{thmABC}, and the fact that $(f(X(k))-f(X'(k)))^2 \preceq A_I^2$ and $(f(X)-f(X'))^2 \preceq A_I^2$, 
\begin{eqnarray*}
m(\theta)'&\le&  \frac{1}{2}\sum_{k=0}^{\infty} \E \Tr\left( \II[I\notin I_1,\ldots,I_k] A_I^2 \left(\frac{e^{\theta f(X)}+e^{\theta f(X')}}{2}\right) \right)\\
&\le& \frac{1}{2}\sum_{k=0}^{\infty} \frac{1}{n} \left(1-\frac{1}{n}\right)^k \left|\left|\sum_{i=1}^n A_i^2\right|\right| \theta m(\theta)\\
&\le&  \frac{1}{2}\sigma^2 \theta m(\theta),
\end{eqnarray*}
so
\[\log(m(\theta))\le \frac{1}{4}\theta^2 \sigma^2,\]
thus by Proposition \ref{mxlapltf},
\[\PP\{\lambda_{max}(\m{H}(Z))\ge t\}\le d\cdot \inf_{\theta>0}\exp\left\{-\theta t + \frac{1}{4}\theta^2 \sigma^2\right\}\le d\cdot\exp\left(-\frac{t^2}{\sigma^2}\right).\]
\end{proof}
Now we prove the general case:
\begin{proof}[Proof for Dobrushin condition]
Let $X,X',X(k),X'(k)$ be defined analogously to the way it is done in the proof of Theorem 4.3 of \cite{Cth}: $X'$ is defined by choosing $I$ uniformly in $[n]$, and then resampling $X_I$ conditioned on the rest (Gibbs sampler), while $X(k), X'(k)$ are defined by choosing $I_k$ uniformly in $[n]$, resampling $X_{I_k}(k-1)$ and $X'_{I_k}(k-1)$ in the greedy coupling way, i.e. $X_{I_k}(k)$ is resampled conditionally on the rest of $X(k-1)$,  $X'_{I_k}(k)$ is resampled conditionally on the rest of $X'(k-1)$, and at the same time, these two conditional distributions are coupled in the maximal coupling (see \cite{Lindvall}).
Property \/\Propp\/ can be proven by induction, verifying \eqref{sumcond} is left to the reader as exercise.

We can write $f(X(k))-f(X'(k))$ as a telescopic sum:
\begin{eqnarray*}
&&f(X(k))-f(X'(k))= \sum_{i=1}^{n}f\left(X_1(k),\ldots,X_i(k),X'_{i+1}(k),\ldots,X'_{n}(k)\right) \\
&&- f\left(X_1(k),\ldots,X_{i-1}(k),X'_{i}(k),\ldots,X'_{n}(k)\right)=: \sum_{i=1}^{n} Z_i(k),
\end{eqnarray*}
We have
\begin{eqnarray*}
m'(\theta)&=&\E \Tr\left(f(X)e^{\theta f(X)}\right)=\E \Tr\left(F(X,X')e^{\theta f(X)}\right)\\
&=&\frac{1}{2}\E \Tr\left(F(X,X')\left(e^{\theta f(X)}-e^{\theta f(X')}\right)\right)\\
&=&\frac{1}{2}\sum_{k=0}^{\infty}\E\Tr\left(\left(f(X(k))-f(X'(k))\right)\left(e^{\theta f(X)}-e^{\theta f(X')}\right)\right)\\
&=&\frac{1}{2}\sum_{k=0}^{\infty}\sum_{i=1}^{n} \E\Tr\left( Z_i(k)  \left(e^{\theta f(X)}-e^{\theta f(X')}\right)\right),\\
\end{eqnarray*}
and obviously $Z_i(k)=L_i(k) Z_i(k)$, so by Theorem \ref{thmABC}, we have
\begin{eqnarray*}
&&\frac{1}{2}\E \Tr\left(L_i(k) Z_i(k) \cdot \left(e^{\theta f(X)}-e^{\theta f(X')}\right)\right)\\
&&\le \frac{1}{2}\E \Tr \left(L_i(k) \theta \frac{1}{4}\left((Z_i(k))^2+ (f(X)-f(X'))^2\right)\cdot \left(e^{\theta f(X)}+e^{\theta f(X')}\right)\right)\\
&&\le \frac{1}{8}\E \Tr \left(L_i(k) \theta\left(A_I^2+A_i^2\right)\cdot \left(e^{\theta f(X)}+e^{\theta f(X')}\right)\right)\\
&&\le \frac{1}{4}\E \Tr \left(l_i(k) \theta\left(A_I^2+A_i^2\right)\cdot e^{\theta f(X)}\right).
\end{eqnarray*}

Let $D$ be the Dobrushin dependence matrix of $X_1,\ldots,X_n$, let us denote $B:=\left(1-\frac{1}{n}\right)E+\frac{1}{n}D$, with $E$ being the $n\times n$ identity matrix. Let $L_i(k):=\II[X_i(k)\ne X'_i(k)]$, and let $l_i(k):=\E(L_i(k)|X,X')$. Page 77-78 of \cite{Cth} proves that $l(k)\le B^k e(I)$, with $e(I)$ denoting the vector whose $I$th coordinate is $1$ and the rest is 0.

\begin{eqnarray*}
&&\sum_{i=1}^n \frac{1}{4}\E \Tr \left(l_i(k) \left(A_I^2+A_i^2\right)\cdot \theta e^{\theta f(X)}\right)\\
&&\le \sum_{i=1}^n \frac{1}{4}\E \Tr \left([B^k e(I)]_i \left(A_I^2+A_i^2\right)\cdot \theta e^{\theta f(X)}\right)\\
&&=\sum_{i=1}^n\sum_{j=1}^n \frac{1}{n}\frac{1}{4}\E \Tr \left([B^k e(j)]_i \left(A_j^2+A_i^2\right)\cdot \theta e^{\theta f(X)}\right)\\
&&=\sum_{j=1}^n \frac{1}{n}\frac{1}{4}\E \Tr \left(\sum_{i=1}^n [B^k e(j)]_i  A_j^2\cdot \theta e^{\theta f(X)}\right)\\
&&+\sum_{i=1}^n \frac{1}{n}\frac{1}{4}\E \Tr \left(\sum_{j=1}^n [B^k e(j)]_i  A_i^2\cdot \theta e^{\theta f(X)}\right)
\end{eqnarray*}
Now \begin{eqnarray*}
&&\sum_{i=1}^n [B^k e(j)]_i= ||B^k e(j)||_1\le \left(||B||_1\right)^k,\text{ and}\\
&&\sum_{j=1}^n [B^k e(j)]_i \le \left(||B||_{\infty}\right)^k,
\end{eqnarray*}
so
\begin{eqnarray*}
\frac{1}{2}\E \Tr\left(L_i(k) Z_i(k) \cdot \left(e^{\theta f(X)}-e^{\theta f(X')}\right)\right)\le \frac{1}{4n}
\sigma^2 \theta m(\theta) \left(\left(||B||_1\right)^k+ \left(||B||_{\infty}\right)^k\right).
\end{eqnarray*}
Summing up in $k$, and noticing that $||B||_1\le 1-\frac{1}{n}+\frac{1}{n}||D||_1$ gives
\[m'(\theta)\le \frac{\theta}{4}\sigma^2 m(\theta)\left(\frac{1}{1-||D||_1}+\frac{1}{1-||D||_{\infty}}\right),\]
and thus we get the result by Proposition \ref{mxlapltf}, as in the independent case.
\end{proof}

\end{proof}

\section{Proof of trace inequalities}

Before starting the proof, we state a few simple inequalities:
\begin{itemize}
\item For any $P,Q\in \MM^d$, we have
\begin{equation}\label{sqrm}
PQ+Q^*P^*\preceq PP^*+Q^*Q,
\end{equation}
which follows from $(P+Q^*)(P^*+Q)\succeq 0$.
\item Also, we can easily prove that if $P,Q,R,S \in \HH^d$, then
\begin{equation}\label{sqrm4}
Re(\Tr\left(PQRS\right))\le \Tr\left(\frac{(P^2+R^2)(Q^2+S^2)}{4}\right)
\end{equation}
To prove this, just apply \eqref{sqrm} to $(PQ)(RS)$ and to $(QR)(SP)$, and rearrange the terms.
\end{itemize}

\begin{proof}[Proof of Theorem \ref{thmABC}]
First notice that adding a constant times identity matrix to $A$ and $B$ multiplies both sides by the same number. Therefore we can suppose without loss of generality that $A,B\succeq 0$.

By taking Taylor expansion, the inequality becomes
\begin{eqnarray*}
\Tr\left(C\sum_{k=1}^{\infty}\frac{A^k-B^k}{k!}\right)\le \Tr\left(\frac{C^2+(A-B)^2}{4}\sum_{k=1}^{\infty}\frac{A^{k-1}+B^{k-1}}{(k-1)!}\right)
\end{eqnarray*} 
To show this, we will prove that the inequality holds for each term in the sums, i.e. we claim:
\begin{equation}\label{CAkBk}
\Tr\left(C(A^k-B^k)\right)\le k \cdot \Tr \left(\frac{C^2+(A-B)^2}{4}\left(A^{k-1}+B^{k-1}\right)\right)
\end{equation}

Now
\[A^k-B^k=A^k-A^{k-1}B+A^{k-1}B-A^{k-2}B^2+\ldots+AB^{k-1}-B^k.\]

The terms in the sum are of the form $A^lB^{k-l}-A^{l-1}B^{k-l+1}=A^{l-1}(A-B)B^{k-l}$.
We claim the following about two 'symmetric' terms from such a sum (which clearly implies \eqref{CAkBk}):
\begin{lemma}\label{lemmaABC}
If $A,B,C$ are Hermitian matrices with $A,B$ positive definite, and $0\le k\le n$ are integers, then
\begin{eqnarray}\label{eqABCk}
&&Re\left(\Tr\left(C\left(A^{k}(A-B)B^{n-k}+A^{n-k}(A-B)B^{k}\right)\right)\right)\\
\nonumber&&\le  \Tr \left(\frac{C^2+(A-B)^2}{2}\left(A^{n}+B^{n}\right)\right).
\end{eqnarray}
\end{lemma}

\begin{proof}
Let us denote $D:=A-B$, then the inequality becomes
\begin{equation}\label{holdereq1}
Re\left(\Tr\left(C\left(A^{k}D B^{n-k}+A^{n-k}D B^{k}\right)\right)\right)\le  \Tr 
\left(\frac{C^2+D^2}{2}\left(A^{n}+B^{n}\right)\right)
\end{equation}
If $k=n/2$, then this follows from \eqref{sqrm4}. Suppose, without loss of generality, that $k<n/2$. Now we can get rid of $C$ in the following way: 
\begin{eqnarray*}
&&Re\left(\Tr\left(CA^{k}D B^{n-k}+CA^{n-k}D B^{k}\right)\right)=\\
&&Re\left(\Tr\left((B^{n/2}C)(A^{k}D B^{n/2-k})+(CA^{n/2}) (A^{n/2-k} D B^{k})\right)\right)\le \\
&&\le \frac{1}{2} Re\left(\Tr\left(C^2 B^{n} + A^{2k} D B^{n-2k} D + C^2 A^{n} + A^{n-2k} D B^{2k}D \right)\right)
\end{eqnarray*}
The terms involving $C$ are the same as on the right hand side of \eqref{eqABCk}, so we need to prove that 
\begin{equation}
\Tr\left(A^{2k} D B^{n-2k} D + A^{n-2k} D B^{2k}D \right)
\le \Tr \left(D^2\left(A^{n}+B^{n}\right)\right).
\end{equation}
Both sides are real so we did not write the real part.

Let us denote, for $0\le l\le n$, 
\begin{eqnarray*}
R:=\Tr \left(D^2\left(A^{n}+B^{n}\right)\right),\text{ and }\\
T_l:=\Tr\left(A^{l} D B^{n-l} D + A^{n-l} D B^{l}D \right).
\end{eqnarray*}
We can show that when $l=n/2$, then $T_l\le R$ holds. Let us denote the maximum of $T_l$ as $T:=\max_{1\le l \le n}T_l$, and suppose that the maximum is taken at $l_0< n/2$, then
\begin{eqnarray*}
&&\Tr\left(A^{l_0} D B^{n-l_0} D + A^{n-l_0} D B^{l_0}D \right)\\
&&=\Tr\left((A^{l_0} D B^{n/2-l_0})(B^{n/2}D)+
(A^{n/2} D) (B^{l_0}D A^{n/2-l_0})\right),
\end{eqnarray*}
so as previously, using \eqref{sqrm} we can show that $T_{l_0}\le \frac{1}{2} T_{2l_0}+\frac{1}{2}R$, i.e. $T\le \frac{1}{2}T+\frac{1}{2}R$, and thus $T\le R$.

For $l_0>n/2$, we can show that $T_{l_0}\le \frac{1}{2} T_{2n-2{l_0}}+\frac{1}{2}R$.
\end{proof}
\end{proof}

\begin{proof}[Proof of Theorem \ref{mxholder}]
\eqref{holdereq2} is a generalization of \eqref{holdereq1}. If $p$ is rational of the form $\frac{a}{b}$, then we can proceed the same way as in the proof of Lemma \ref{lemmaABC}, with $A^{\frac{1}{b}}$ and $B^{\frac{1}{b}}$ taking the role of $A$ and $B$. If $p$ is irrational, we can write it as the limit of rationals, and use continuity to get the result.
\end{proof}

\section{Open problems}
Based on extensive numerical evidence and on some theoretical results, we conjecture the following trace inequalities:

\begin{enumerate}
\item Let $A,B,C\in \HH^d$, then
\begin{eqnarray}\label{expconj}
&&\Tr\left(C(e^A-e^B)\right)\\
\nonumber&&\le \Tr\left(\frac{C_+^2+(A-B)_+^2}{2}e^{A}+\frac{C_-^2+(A-B)_-^2}{2}e^{B}\right)
\end{eqnarray}
\item We expect \eqref{expconj} to generalize to any monotone increasing convex function $f$, i.e. we expect that in such situations,
\begin{eqnarray}\label{fconj}
&&\Tr\left(C(f(A)-f(B))\right)\\
\nonumber&&\le \Tr\left(\frac{C_+^2+(A-B)_+^2}{2}f'(A)+\frac{C_-^2+(A-B)_-^2}{2}f'(B)\right)
\end{eqnarray}
\item \eqref{expconj} would imply concentration for self-bounding matrix valued functions (in the sense of \cite{LugosiSelfBounding}). A similar setting has been already studied in \cite{Mackeythesis}, Theorem 25, however, this theorem requires a very strong self-reciprocity condition, which may not be satisfied in general.

We define matrix self-bounding functions as follows: 
\begin{definition}
An $\HH^d$ valued function $\m{H}(Z_1,\ldots,Z_n)$ is said to be $(a,b)$ matrix self-bounding, if for any $Z_1',\ldots,Z_n'$,
\begin{enumerate}
\item $\m{H}(Z)-\m{H}(Z_1,\ldots,Z_i',\ldots,Z_n)\le \m{I}_d$, and
\item $\sum_{i=1}^n(\m{H}(Z)-\m{H}(Z_1,\ldots,Z_i',\ldots,Z_n))_+\preceq a \m{H}(Z)+b \m{I}_d$.
\end{enumerate}
An $\HH^d$ valued function $\m{H}(Z_1,\ldots,Z_n)$ is said to be weakly $(a,b)$ matrix self-bounding, if for any $Z_1',\ldots,Z_n'$, 
\[\sum_{i=1}^n(\m{H}(Z)-\m{H}(Z_1,\ldots,Z_i',\ldots,Z_n))_+^2\preceq a \m{H}(Z)+ b \m{I}_d.\]
\end{definition}
\end{enumerate}

We expect concentration inequalities similar to Theorem 1 of \cite{LugosiSelfBounding} to hold for such functions.

\section*{Acknowledgements}
The author thanks Doma Sz\'{a}sz and Mogyi T\'{o}th for infecting him with their enthusiasm of probability.
He thanks his thesis supervisors, Louis Chen and Adrian R\"{o}llin, for the opportunity to study in Singapore, and their  useful advices. Finally, many thanks to my brother, Roland Paulin, for the enlightening discussions.
\bibliographystyle{imsart-nameyear}
\bibliography{References}

\end{document}